\documentclass[11pt,a4paper,reqno, draft]{amsart}
\usepackage{amsmath, amssymb, latexsym, enumerate,  graphicx,tikz, stmaryrd, eufrak,enumitem,MnSymbol}

\usepackage{thmtools}

\usepackage[pdftex]{hyperref}
\usepackage{cleveref}
\usepackage{bbm}
\usepackage{cases}
\usepackage{array}
\usepackage{tikz-cd}
\usepackage[all]{xy}

\usepackage[hmarginratio={1:1},vmarginratio={1:1},lmargin=80.0pt,tmargin=90.0pt]{geometry}

\usepackage{tikz}
\tikzset{anchorbase/.style={baseline={([yshift=-0.5ex]current bounding box.center)}}}
\usetikzlibrary{decorations.markings}
\usetikzlibrary{decorations.pathreplacing}
\usetikzlibrary{arrows,shapes,positioning,backgrounds}
\tikzstyle directed=[postaction={decorate,decoration={markings,
    mark=at position #1 with {\arrow{>}}}}]
\tikzstyle rdirected=[postaction={decorate,decoration={markings,
    mark=at position #1 with {\arrow{<}}}}]
    
\usepackage{graphicx}
\usepackage[vcentermath]{youngtab}
\usepackage{nicefrac}

 \newlength{\baseunit}               
 \newcount{\numlines}                
\newtheorem{theorem}[subsubsection]{Theorem}
\newtheorem{lemma}[theorem]{Lemma}
\newtheorem{prop}[theorem]{Proposition}
\newtheorem{corollary}[subsubsection]{Corollary}

\theoremstyle{definition}
\newtheorem{definition}[subsubsection]{Definition}

\newtheorem{remark}[theorem]{Remark}

\newtheorem{example}[subsubsection]{Example}

\newcommand{\bA}{\mathbf{A}}
\newcommand{\bB}{\mathbf{B}}
\newcommand{\bC}{\mathbf{C}}
\newcommand{\bD}{\mathbf{D}}
\newcommand{\bI}{\mathbf{I}}
\newcommand{\bJ}{\mathbf{J}}

\newcommand{\cV}{\mathcal{V}}

\newcommand{\Set}{\mathsf{Set}}

\newcommand{\Rep}{\mathsf{Rep}}

\newcommand{\tto}{\twoheadrightarrow}
\newcommand{\cO}{\mathcal{O}}

\newcommand{\mN}{\mathbb{N}}
\newcommand{\mZ}{\mathbb{Z}}
\newcommand{\mC}{\mathbb{C}}

\newcommand{\End}{\mathrm{End}}
\newcommand{\Ob}{\mathrm{Ob}}
\newcommand{\Ext}{\mathrm{Ext}}
\newcommand{\colim}{\mathrm{colim}}
\newcommand{\Hom}{\mathrm{Hom}}

\newcommand{\op}{\mathrm{op}}
\newcommand{\Ind}{\mathrm{Ind}}

\newcommand{\Pro}{\mathrm{Pro}}

\begin{document}
\title[Homological properties of highest weight categories]{Some homological properties of ind-completions and highest weight categories}
\author{Kevin Coulembier}
\address{K.C.: School of Mathematics and Statistics, University of Sydney, NSW 2006, Australia}
\email{kevin.coulembier@sydney.edu.au}



\begin{abstract}
We demonstrate equivalence between two definitions of lower finite highest weight categories. We also show that, in the presence of a duality, a lower finite highest weight structure on a category is unique. Finally, we give a proof for the known fact that any abelian category is extension full in its ind-completion.
\end{abstract}

\maketitle


\section*{Introduction}
In \cite{blocks} it was proved that, when the category of finite dimensional modules of a finite dimensional associative algebra admits a simple preserving duality, it is a highest weight category in at most one way (up to equivalence). This provided a crucial tool for the classification of blocks up to equivalence in category $\cO$. Note that on the other extreme we have hereditary algebras, for which the module category admits many highest weight structures, see \cite[Theorem~1]{DR1}.

Most highest weight categories in nature exhibit some finiteness. We say a highest weight category is finite if it is of the above module category type.
In \cite{BS} a comprehensive study of 3 types of `semi infinite' highest weight categories is undertaken: upper finite, essentially finite and lower finite highest weight categories. It is observed in \cite{blocks} that the uniqueness result of the finite case extends canonically to the upper finite case, but is false for the essentially finite case.

This paper is motivated by the remaining question, left open in \cite{blocks}, of whether lower finite highest weight categories with a simple preserving duality enjoy uniqueness of highest weight structure. The proof in \cite{blocks} clearly does not extend to this setting. However, we provide an alternative proof which is applicable to lower finite categories. The crucial idea is to obtain a criterion for when a standard object is simple, which is `Koszul dual' to the criterion in \cite{blocks} for which standard objects are projective. Once this is observed, the proof purely relies on typical highest weight arguments. As an example, it follows that the category of algebraic representations of a reductive group has precisely one highest weight structure.

Motivated by the above, we prove that the definition of lower finite highest weight categories in \cite{BS} is equivalent to the definition in \cite{RW}. Since we need to work partially in the ind-completion we also provide a proof that an abelian category is extension full in its ind-completion. This fact is known and included in a more general observation about the derived category in \cite[Chapter~15]{KS}.

Finally, we address the following question, asked to us by Pramod Achar. Can the ind-completion of a lower finite highest weight category with simple preserving duality be a highest weight category in the sense of \cite{CPS} (that is, without any finiteness assumption), in any other way than via its (unique) structure of a lower finite highest weight category? As we show, the answer is `no'.

The paper is organised as follows. In Section~\ref{Prel} we recall some background. In Section~\ref{SecExt3} we obtain a simple criterion on an abelian category with abelian subcategory, under which all extension groups from any object in the big category to any object in the small category are colimits of extension groups in the small category. We show that this applies to ind-completions, which in particular implies extension fullness. In Section~\ref{SecDef} we compare the definitions of \cite{BS, RW}. We exploit this to derive all properties of such categories that we will need, in a self-contained way. In Section~\ref{SecUniq} we prove the uniqueness of highest weight structure. In Section~\ref{SecExam} we discuss some examples.

\section{Preliminaries}\label{Prel}
We set $\mN=\{0,1,2,\ldots\}$.

\subsection{Abelian categories}

Let $\bA$ be an abelian category. Since we take the convention that any algebraic structure is a set (and not a proper class), abelian categories are automatically assumed locally small.

\begin{definition}\label{DefAbSub}
\begin{enumerate}
\item An {\bf abelian subcategory} of $\bA$ is a full replete subcategory $\bB$ which is again abelian and such that the inclusion $\bB\to\bA$ is exact. Equivalently, an abelian subcategory is a full additive subcategory  $\bB$ which contains all kernels and cokernels in $\bA$ of morphisms in $\bB$.
\item A full additive subcategory $\bB$ of $\bA$ is {\bf topologising} if it is closed under taking subquotients in $\bA$. In particular a topologising subcategory is an abelian subcategory.
\item A {\bf Serre subcategory} $\bB$ of $\bA$ is a full subcategory such that  for every short exact sequence $0\to X\to Y\to Z\to 0$ in $\bA$, we have $Y\in\bB$ if and only if $X,Z\in \bB$. In particular Serre subcategories are topologising.
\end{enumerate}
\end{definition}

\subsubsection{Extensions}\label{SecExt1} We introduce Yoneda extensions following \cite[III.3.2]{Verdier}.
Let $\bA$ be an abelian category, $M,N\in\bA$ and $d\in\mathbb{Z}_{> 0}$. We denote by $E^d(X,Y)$ the class of exact sequences
\begin{equation}\label{eqX}
\xymatrix{ 
\mathcal{X}:\;0\ar[r]&N\ar[r]&X^1\ar[r]&X^2\ar[r]&\dots\ar[r]&X^d\ar[r]&M\ar[r]&0,
}
\end{equation}
in $\bA$. Let $\sim$ be the binary relation on $E^d(M,N)$, such that $\mathcal{X}\sim\mathcal{Y}$ if there exists a commutative diagram in $\bA$ with exact rows
\begin{displaymath}
\xymatrix{ 
\mathcal{X}:&0\ar[r]&N\ar[r]\ar@{=}[d]&X^1\ar[r]&X^2\ar[r]&\dots\ar[r]&X^d\ar[r]&M\ar[r]\ar@{=}[d]&0\\
&0\ar[r]&N\ar[r]\ar@{=}[d]&Z^1\ar[r]\ar[d]\ar[u]&Z^2\ar[r]\ar[d]\ar[u]&\dots\ar[r]&Z^d\ar[r]\ar[d]\ar[u]&M\ar[r]\ar@{=}[d]&0\\
\mathcal{Y}:&0\ar[r]&N\ar[r]&Y^1\ar[r]&Y^2\ar[r]&\dots\ar[r]&Y^d\ar[r]&M\ar[r]&0.
}
\end{displaymath}

By \cite[Proposition 3.2.2]{Verdier}, the relation $\sim$ is an equivalence relation on $E^d(M,N)$ and we define $\Ext^d_{\bA}(M,N)$ as $E^d(M,N)/\hspace{-1.2mm}\sim$. If $\bA$ is essentially small or has enough projective or injective objects, then $\Ext^d_{\bA}(M,N)$ is a set. For every category we will encounter, at least one of these conditions will be satisfied, so ignore the fact that $\Ext^d_{\bA}(M,N)$ might be a proper class.

Then $\mathrm{Ext}_{\bA}^d(M,N)$ has the natural structure of an abelian group via the Baer sum and is functorial in both variables. For example, under a morphism $f:N\to N'$ the extension represented by \eqref{eqX} is sent to the element of $\Ext^d_{\bA}(M,N')$ represented by
$$\xymatrix{ 
0\ar[r]&N'\ar[r]&X^1\sqcup_{N}N'\ar[r]&X^2\ar[r]&\dots\ar[r]&X^d\ar[r]&M\ar[r]&0.
}$$

\subsubsection{Extension fullness}\label{SecExt} Let $\bA$ be an abelian category with abelian subcategory $\bB$.
As the inclusion functor $\boldsymbol{\iota}:\bB\to \bA$ is exact, for every
$M,N\in\bB$ and $d\in\mN$, it induces a group homomorphism 
\begin{equation}\label{DefIota}
\iota^d_{M,N}:\mathrm{Ext}_{\bB}^d(M,N)\to \mathrm{Ext}_{\bA}^d(M,N).
\end{equation}
 In general, these homomorphisms $\iota^d_{M,N}$ are neither injective nor surjective. Of course, $\iota^0_{M,N}$ is always bijective and $\iota^1_{M,N}$ is always injective.

We say that $\bB$ is {\bf extension full} in $\bA$ provided that the homomorphisms $\iota^d_{M,N}$ are bijective for all
$d\in\mN$ and all $M,N\in\bB$.



\subsubsection{Serre quotient categories}

We recall some results from \cite[Chapitre~III]{Gabriel}. For an abelian category $\bA$ with Serre subcategory~$\bB\subset\bA$,
the Serre quotient category~$\bA/\bB$ is defined by setting~$\Ob(\bA/\bB):=\Ob \bA$, and for $X,Y\in\bA$
\begin{equation}\label{HomQuo}\Hom_{\bA/\bB}(X,Y)\;:=\;\varinjlim \Hom_{\bA}(X',Y/Y'),\end{equation}
where~$X'$, resp. $Y'$, runs over all subobjects in~$\bA$ (ordered by inclusion) of~$X$, resp. $Y$, such that~$X/X'\in\bB\ni Y'$. For the definition of the composition of two morphisms in~$\bA/\bB$, we refer to \cite{Gabriel}.

By \cite[Proposition~III.1.1]{Gabriel}, the category~$\bA/\bB$ is abelian and we have an exact functor~$\pi:\bA\to\bA/\bB$, which is the identity on objects and is given on morphisms by the canonical morphism from
$\Hom_{\bA}(X,Y)$ to $\varinjlim \Hom_{\bA}(X',Y/Y').$

\subsection{Ind-completion} We refer to \cite{KS} for a complete treatment on ind-completions.

\subsubsection{} For a locally small category $\bC$, we denote its ind-completion by $\Ind\bC$. This is the full subcategory of $[\bC^{\op},\Set]$ of objects isomorphic to `formal colimits' of functors $F:\bI\to\bC$, from small filtered categories $\bI$. The formal colimit of $F$ is the colimit of the composition $\bI{\to}\bC{\to}[\bC^{\op},\Set]$ with the Yoneda embedding. For $F:\bI\to\bC$ and $G:\bJ\to\bC$, it thus follows from the Yoneda lemma that in $\Ind\bC$:
\begin{equation}\label{eqlimlim}\Hom(\colim F,\colim G)\;=\; \mathrm{lim}_{i\in \bI} \colim_{j\in\bJ} \Hom_{\bC}(F(i),G(j)).\end{equation}
In particular, $\bC$ is the full subcategory of compact objects in $\Ind\bC$.

In case we consider direct limits or inverse limits (colimits or limits of functors from a directed poset $I$ or its inverse, seen as a category), we will use the notation $\varinjlim_{i\in I}$ or $\varprojlim_{i\in I}$.


\subsubsection{}\label{indab}
If $\bC$ is abelian, it is well-known that $\Ind\bC$ is also abelian and that $\bC$ is an abelian subcategory. For two functors $F,G:\bI\to\bC$ and a natural transformation $\alpha: F\Rightarrow G$, the (co)kernel of $\colim(\alpha)$ is given by the colimit of the (co)kernel of $\alpha$. Moreover, any morphism in $\Ind\bC$ can be written as such a $\colim(\alpha)$.

In case $\bC$ is also essentially small, the category $\Ind\bC$ is a Grothendieck category, since the isomorphism classes of objects in $\bC$ constitute a set of generators.



\subsection{Finite length categories}\label{SFL}
Fix a field $k$. 

\subsubsection{} A {\bf finite length category} is an essentially small $k$-linear abelian category for which all objects have finite length and all morphism spaces are finite dimensional. Any Serre quotient or subcategory of a finite length category is again of finite length.

A {\bf finite category} is a finite length category $\bA$ with an object $G$ such that $\bA$ is equal to its topologising subcategory generated by $G$. This is the same as saying that $\bA$ is equivalent to the category of finite dimensional modules of a finite dimensional associative algebra, or that it is a finite length category with finitely many simple objects and enough projective (or yet equivalently: injective) objects, see \cite[2.14-2.17]{Del90}.

\begin{lemma}[Takeuchi]\label{LemTak}
A $k$-linear abelian category $\bA$ is a finite length category if and only if there exists a coalgebra $C$ over $k$ such that $\bA$ is equivalent to the category of finite dimensional right $C$-modules {\rm{comod-}}$C$. Furthermore, the category of all right $C$-modules {\rm{Comod-}}$C$ is then equivalent to $\Ind\bA$.
\end{lemma}
\begin{proof}
For any coalgebra $C$, the category comod-$C$ is clearly finite length.

Now let $\bA$ be a finite length category. The category $\Ind\bA$ is `locally finite' in the sense of \cite[Definition~4.1]{Tak}. The only property to check is that for $M\in\Ind \bA$ and a directed set of subobjects $M_i\subset M$, the canonical morphism $\varinjlim M_i\to M$ is a monomorphism. 
Furthermore, $\Ind\bA$ satisfies the equivalent properties in \cite[Proposition~4.3]{Tak}, so it is `of finite type'. By \cite[Theorem~5.1]{Tak} there exists a coalgebra $C$ such that $\Ind\bA$ is equivalent to Comod-$C$. Since comod-$C$ is precisely the subcategory of compact objects in Comod-$C$, the former is equivalent to $\bA$.
\end{proof}

\section{Extension fullness in the ind-completion}
\label{SecExt3}

\subsection{General principles}
Consider an abelian category $\bA$ with abelian subcategory $\bB$. 

\subsubsection{}For each $X\in\bA$, we have the category $\bI_X$ of pairs $(X',f)$ with $X'\in \bB$ and $f:X'\to X$ a morphism in $\bA$. The morphisms in $\bI_X$ are given by morphisms in $\bB$ which yield a commutative diagram in the obvious way. In other words $\bI_X$ is a full subcategory of the slice category $\bA/X$. We have the forgetful functor $F_X:\bI_X\to\bB$, $(X',f)\mapsto X'$. Note that the fact that $\bB$ is an {\em abelian} subcategory of $\bA$ easily implies that $\bI_X$ is filtered.

\begin{theorem}\label{Thm0}
The following conditions on $(\bA,\bB)$ are equivalent.
\begin{enumerate}
\item For every epimorphism $p:A\tto B$ in $\bA$, with $B\in \bB$, there exists a morphism $f:B'\to A$ with $B'\in \bB$ such that $p\circ f$ is still an epimorphism.
\item For each $B\in \bB$ and $X\in  \bA$, the canonical homomorphism 
$$\phi_1:\,\colim\Ext^1_{\bB}(B,F_X)\;\to\;\Ext^1_{\bA}(B,X)$$
is an epimorphism.
\item For each $B\in \bB$, $X\in  \bA$ and $i\in\mN$, the canonical homomorphism 
$$\phi_i:\,\colim\Ext^i_{\bB}(B,F_X)\;\to\;\Ext^i_{\bA}(B,X)$$
is an isomorphism.
\end{enumerate}
\end{theorem}
\begin{proof}
Obviously, (3) implies (2). Now we assume that (2) is satisfied and prove (1). Suppose we are given an epimorphism $p:A\tto B$ in $\bA$, with $B\in\bB$. Set $K=\ker(p)\in\bA$ and denote by $x$ the corresponding element in $\Ext^1_{\bA}(B,K)$. By assumption (2) for $X=K$ there exists a morphism $K'\to K$ with $K'$ in $\bB$ such that $x$ is in the image of
$$\Ext^1_{\bB}(B,K')\hookrightarrow \Ext^1_{\bA}(B,K') \to  \Ext^1_{\bA}(B,K),\quad E\mapsto E\sqcup_{K'}K, $$
where we abbreviated the morphism to its action on the middle term of the short exact sequence. In particular, there exists $E\in\bB$ and a morphism $E\to E\sqcup_{K'}K\stackrel{\sim}{\to}A$ which yields an epimorphism $E\to A\to B$ as desired.

Finally, we prove that (1) implies (3).
We fix arbitrary $B\in\bB$ and $X\in\bA$. That $\phi_0$ is always an isomorphism is a standard exercise, so we can fix $i>0$.

First we prove that $\phi_i$ is surjective.
Consider an exact sequence in $\bA$
\begin{equation}\label{extseq}
\xymatrix@C=5mm{
0\ar[r]&X\ar[r]&C^1\ar[r]&C^2\ar[r]&\;\dots\;\ar[r]&C^{i-1}\ar[r]&C^i\ar[r]&B\ar[r]&0.\\
}
\end{equation}
We also set  $C^0=X$, $C^{i+1}=B $ and $C^{i+2}=0$. There exist a minimal $j\in\mN$ such that $C^k\in\bB$ for $k>j$. We intend to prove that the extension represented by \eqref{extseq} is in the image of $\phi_i$. If $j=0$ there is nothing to prove, 
so the only non-trivial cases are $0<j\le i$.

We let $Z^{k}$ denote the kernel of $C^{k}\to C^{k+1}$, for $1\le k\le i+1$. In particular $Z^{i+1}=B$. By assumption $Z^{j+1}\in\bB$. By exactness of \eqref{extseq}, we have $C^j\tto Z^{j+1}$ and hence by assumption (1) there exists $T\in\bB$ with $T\to C^j$ such that the composite $T\to C^j\to Z^{j+1}$ is still an epimorphism. We denote the kernel of the latter epimorphism by $K\in\bB$. The morphism $K\hookrightarrow T\to C^j$ thus factors as $K\to Z^j\hookrightarrow C^j$. This yields a commutative diagram with exact rows
$$\xymatrix{
0\ar[r]&Z^{j-1}\ar[r]&C^{j-1}\ar[r]& Z^{j}\ar[r]&0\\
0\ar[r]&Z^{j-1}\ar@{=}[u]\ar[r]& C^{j-1}\times_{Z^j}K\ar[r]\ar[u]&K\ar[r]\ar[u]&0.
}$$
With $P:=C^{j-1}\times_{Z^j}K$, this allows to construct a commutative diagram with exact rows
\begin{displaymath}
\xymatrix@C=5mm{
0\ar[r]&C^0\ar[r]&C^1\ar[r]&\;\dots\; \ar[r]& C^{j-2}\ar[r]& C^{j-1}\ar[r]& C^{j}\ar[r] & C^{j+1}\ar[r] &\;\dots\;  \ar[r]&C^i\ar[r]&B\ar[r]&0\\
0\ar[r]&C^0\ar@{=}[u]\ar[r]&C^1\ar[r]\ar@{=}[u]&\;\dots\; \ar[r]& C^{j-2}\ar[r]\ar@{=}[u]& P\ar[u]\ar[r]&T\ar[u]\ar[r] & C^{j+1}\ar[r] \ar@{=}[u]&\;\dots\;  \ar[r]&C^i\ar@{=}[u]\ar[r]&B\ar[r]\ar@{=}[u]&0.\\
}
\end{displaymath}
If we had $j=1$, then the above sequence is in $\bB$ and therefore $P\to C^0=X$ gives an object of $\bI_X$ such that that the original extension represented by \eqref{extseq} is in the image of $\Ext^i_{\bB}(B,P)\to\Ext^i_{\bA}(B,X)$.
On the other hand, if $j>1$, then we obtain that \eqref{extseq} is equivalent under $\sim$ of \ref{SecExt1} to a sequence in $E^i(B,X)$ where all objects in degrees strictly higher than $j-1$ belong to $\bB$ and we can iterate the above procedure. This proves that 
$\phi_i$
is surjective.

Now we prove that $\phi_i$ is injective. We start with two observations (I) and (II):

(I) Consider a morphism $f:X'\to X''$ in $\bB$. If there exists a commutative diagram in $\bB$ with exact rows
$$
\xymatrix@C=5mm{
0\ar[r]&X'\ar[r]\ar[d]^f&D^1\ar[r]\ar[d]&D^2\ar[r]\ar[d]&\dots\ar[r]&
D^{i-1}\ar[r]\ar[d]&D^i\ar[r]\ar[d]&B\ar[r]\ar@{=}[d]&0\\
0\ar[r]&X''\ar[r]&C^1\ar[r]&C^2\ar[r]&\dots\ar[r]&C^{i-1}\ar[r]&C^i\ar[r]&B\ar[r]&0
}
$$
then the homomorphism $\Ext^i_{\bB}(B,f)$ will send the extension represented by the upper row, to the extension represented by the lower row. Indeed, this is an immediate application of the universality of the pushout $D^1\sqcup_{X'}X''$. 

(II) Assume we are given a commutative diagram in $\bA$ with exact rows
$$
\xymatrix@C=5mm{
0\ar[r]&X'\ar[r]\ar[d]&D^1\ar[r]\ar[d]&D^2\ar[r]\ar[d]&\dots\ar[r]&
D^{i-1}\ar[r]\ar[d]&D^i\ar[r]\ar[d]&B\ar[r]\ar@{=}[d]&0\\
0\ar[r]&X\ar[r]&C^1\ar[r]&C^2\ar[r]&\dots\ar[r]&C^{i-1}\ar[r]&C^i\ar[r]&B\ar[r]&0\\
0\ar[r]&X''\ar[r]\ar[u]&E^1\ar[r]\ar[u]&
E^2\ar[r]\ar[u]&\dots\ar[r]&E^{i-1}\ar[r]\ar[u]&E^i\ar[r]\ar[u]&B\ar[r]\ar@{=}[u]&0
}
$$
where the upper and lower rows are in $\bB$. Then we also have a diagram as the above, except that the middle row is also in $\bB$ and the first column is a diagram $X'\to Y\leftarrow X''$ over $X$ (so a diagram in $\bI_X$), for some $Y\in\bB$ over $X$, where $X'$ and $X''$ are over $X$ via the original diagram. Indeed, this can be achieved using the exact same iterative procedure as in our proof that $\phi_i$ is surjective. The only adjustment which is required is replacing the object $T\in\bB$ as defined above by $T\oplus D^i\oplus E^i\in\bB$.

Now consider objects $X_1,X_2$ in $\bB$ over $X$ (so objects of $\bI_X$) and $a\in\Ext^i_{\bB}(B,X_1)$ and $c\in\Ext^i_{\bB}(B,X_2)$ represented by
$$0\to X_1\to A^1\to \cdots \to A^i\to B\to 0\;\mbox{ and }\;0\to X_2\to C^1\to \cdots\to C^i\to B\to 0.$$
Assume that $a,c$ are sent to the same element in $\Ext^i_{\bA}(B,X).$ By definition, this means there exists a commutative diagram in $\bA$ with exact rows
\begin{displaymath}
\xymatrix{ 
0\ar[r]&X_1\ar[r]\ar[d]&{A}^1\ar[r]\ar[d]&A^2\ar[r]\ar@{=}[d]&\dots\ar[r]&A^i\ar[r]\ar@{=}[d]&B\ar[r]\ar@{=}[d]&0\\
0\ar[r]&X\ar[r]\ar@{=}[d]&\widetilde{A}^1\ar[r]&A^2\ar[r]&\dots\ar[r]&A^i\ar[r]&B\ar[r]\ar@{=}[d]&0\\
0\ar[r]&X\ar[r]\ar@{=}[d]&E^1\ar[r]\ar[d]\ar[u]&E^2\ar[r]\ar[d]\ar[u]&\dots\ar[r]&E^i\ar[r]\ar[d]\ar[u]&B\ar[r]\ar@{=}[d]&0\\
0\ar[r]&X\ar[r]&\widetilde{C}^1\ar[r]&C^2\ar[r]&\dots\ar[r]&C^i\ar[r]&B\ar[r]\ar@{=}[d]&0\\
0\ar[r]&X_2\ar[u]\ar[r]&C^1\ar[r]\ar[u]&C^2\ar[r]\ar@{=}[u]&\dots\ar[r]&C^i\ar[r]\ar@{=}[u]&B\ar[r]&0,
}
\end{displaymath}
where $\widetilde{A}^1$ and $\widetilde{C}^1$ are pushouts. As in the proof of surjectivity of $\phi_i$, we can replace the middle row by one contained in $\bB$, and the term $X$ is replaced by an object in $\bB$ over $X$ by construction. Then we can use observation (II) to replace the second and forth row by rows contained in $\bB$. The resulting diagram is thus entirely contained in $\bB$ and all morphisms in the first column, say
$$X_1\to X_4\leftarrow X_3\to X_5\leftarrow X_2$$ are canonically over $X$. Observation (I) implies that the homomorphisms 
$$\Ext^1_{\bB}(B,X_1)\;\to\;\Ext^1_{\bB}(B,X_4\sqcup_{X_3}X_5)\;\leftarrow\; \Ext^1_{\bB}(B,X_2),$$
induced by the canonical morphisms $X_i\to X_4\sqcup_{X_3}X_5$ over $X$, for $i\in\{1,2\}$ 
 will send the extensions $a,c$ to the same element in the middle group. This shows that $\phi_i$ is injective.
\end{proof}

The following theorem is a generalisation of \cite[Proposition~5]{Sigma}.

\begin{theorem}\label{Thm1}
 If for every epimorphism $p:A\tto B$ in $\bA$, with $B\in \bB$, there exists a morphism $f:B'\to A$ with $B'\in \bB$ such that $p\circ f$ is still an epimorphism, then $\bB$ is extension full in $\bA$.
\end{theorem}
\begin{proof}
Take $M,N\in\bB$ and recall $\iota^i_{M,N}$ from \eqref{DefIota}. Then $\phi_i$ from \ref{Thm0}(3) is the composite of
$$\colim\Ext^i_{\bB}(M,F_N)\;\tto\;\Ext^i_{\bB}(M,N)\;\stackrel{\iota^i_{M,N}}{\to}\;\Ext^i_{\bA}(M,N).$$
By Theorem~\ref{Thm0}, this composite is an isomorphism, which forces $\iota^i_{M,N}$ to be an isomorphism as well.
\end{proof}

\begin{remark}
The statement in Theorem~\ref{Thm1} can not be reversed. Typically the extension full subcategories of Corollary~\ref{CorOmegaEF} below, will not satisfy property \ref{Thm0}(1).
\end{remark}


\subsection{Homological properties of ind-completions}

\begin{theorem}\label{ThmInd}
Let $\bC$ be an abelian category.
\begin{enumerate}
\item The category $\bC$ is extension full in $\Ind\bC$ and $\Pro\bC$.
\item For a functor $F:\bI\to\bC$, with $\bI$ filtered, we have a canonical isomorphism
$$\psi_i:\, \colim\Ext^i_{\bC}(Z,F)\;\stackrel{\sim}{\to}\;\Ext^i_{\Ind\bC}(Z,\colim F),$$
for every $i\in\mN$ and $Z\in \bC$.
\end{enumerate}
\end{theorem}

\begin{remark}
Theorem~\ref{ThmInd}(1) can also be extracted from the more general \cite[Theoren~15.3.1]{KS}.
Theorem~\ref{ThmInd}(1) shows in particular that all extensions in essentially small abelian categories can be defined in terms of injective or projective resolutions, see \cite[Th\'eor\`eme II.2]{Gabriel}. We will henceforth write extension groups in $\Ind\bC$ simply as $\Ext^\bullet_{\bC}$.
\end{remark}
We start the proof with the following standard observation.

\begin{lemma}\label{LemInd}
The abelian subcategory $\bC$ of $\Ind\bC$ satisfies property~\ref{Thm0}(1).
\end{lemma}
\begin{proof}
Take an epimorphism $p:\colim F\tto B$ in $\Ind\bC$, for $B\in\bC$ and some functor $F:\bI\to\bC$ with $\bI$ filtered. Consider the cokernel $Q$ of the natural transformation from $F$ to the constant functor $\bI\to\bC$ at $B$.
By \ref{indab}, we find $\colim Q=0.$ Using the fact that $B$ is compact then shows that the canonical epimorphism $B\to Q(i_0)$ is zero for some $i_0\in\bI$ and hence $Q(i_0)\simeq 0$. Consequently $F(i_0)\to \colim F\tto B$ is an epimorphism.
\end{proof}





\begin{proof}[Proof of Theorem~\ref{ThmInd}]
For part (1), it suffices to consider $\Ind\bC$, since $(\Pro\bC)^{\op}$ is equivalent to $ \Ind(\bC^{\op})$. The statement thus follows from Theorem~\ref{Thm1} and Lemma~\ref{LemInd}.

Now we will prove part (2). Set $Y:=\colim F$.

First we prove that $\psi_i$ is surjective. Consider $x\in \Ext^i(Z,Y)$. Since $\phi_i$ in \ref{Thm0}(3) is surjective, there exists a morphism $f:Y'\to Y$ with $Y'\in\bC$ such that $x$ is in the image of $\Ext^i_{\bC}(Z,f)$. Since $Y'$ is compact, $f$ must factor through some morphism $Y'\to  F(j)$ and hence $x$ is in the image of $ \Ext^i(Z, F(j))\to\Ext^i(Z,Y)$.

Now we prove that $\psi_i$ is injective. Take $z\in \Ext^i(Z, F(j))$ which is sent to zero in $\Ext^i(Z,Y)$. Since $\phi_i$ in \ref{Thm0}(3) is injective, the canonical morphism $ F(j)\to Y$ must factor through some $ F(j) \to Y''$ in $\bC$ such that $z$ is sent to zero in $\Ext^i(Z,Y'')$. Again, since $Y''$ is compact, the corresponding $Y''\to Y$ factors as $Y''\to F(l)$ for some $l\in\bI$ and hence $z$ is also zero in $\Ext^i(Z,F(l))$ and therefore zero in $\colim\Ext^i(Z, F)$. This concludes the proof.
\end{proof}

\subsection{Some (non)examples} Fix a field $k$.

\begin{theorem}
For a coalgebra $C$, {\rm{comod-}}$C$ is extension full in {\rm Comod-}$C$.

In particular, the category $\Rep G$ of finite dimensional algebraic representations of an affine group scheme $G/k$ is extension full in the category $\Rep^\infty G$ of all algebraic representations.
\end{theorem}
\begin{proof}
This is a direct application of Lemma~\ref{LemTak} and Theorem~\ref{Thm1}.
\end{proof}

\subsubsection{}\label{DefMN}Despite the result in Theorem~\ref{ThmInd}(2), the analogue of equation~\eqref{eqlimlim} does not apply to extension groups. We provide a concrete counterexample.

Let $\bC$ be $\Rep\mathbb{G}_a$, the category of finite dimensional algebraic representations of the additive group $\mathbb{G}_a$ of $k$. In other words, $\Ind\bC$ (respectively $\bC$) is the category of all (respectively all finite dimensional) comodules of the coalgebra $k[x]$, with comultiplication $\Delta(x)=x\otimes 1+1\otimes x$. Note also that a comodule of $k[x]$ is the same as a module of the algebra $k[\theta]$ on which $\theta$ acts nilpotently.

For $i\in\mN$, we denote the $i$-dimensional indecomposable module in $\bC$ by $M_i$, let $M$ in $\Ind\bC$ be the direct limit of
\begin{equation}\label{defM}M_1\hookrightarrow M_2\hookrightarrow M_3\hookrightarrow \cdots,\end{equation}
which means that $M\simeq k[x]$ is the injective hull of $M_1\simeq k$, and we set $N=\oplus_iM_i$.

\begin{lemma}
With notation of \ref{DefMN}, the canonical morphism
$$\Ext^1_{\bC}(M,N)\;\to\;\varprojlim\Ext^1_{\bC}(M_i,N)$$
is an epimorphism, but not a monomorphism.
\end{lemma}
\begin{proof}
We choose a basis $\{e_j\,|\, j>0\}$ of $M$ such that $e_i\in M_i\subset M$. We also introduce the $k$-linear morphism $\rho:M\to k$ by $\rho(e_j)=\delta_{j1}$.
For any vector space $V$ and for $X\in\Ind\bC$, we have an isomorphism
$$\Hom_{\bC}(X,M\otimes_k V)\;\stackrel{\sim}{\to}\;\Hom_k(X,V),\quad\mbox{given by}\quad f\mapsto (\rho\otimes V)\circ f.$$
Indeed, by application of \eqref{eqlimlim}, this claim can be reduced to the case where $X$ is finite dimensional and $V=k$, in which case it is obvious.

Consider an injective coresolution 
$$0\to N\to M\otimes_k V\to M\otimes_k V\to0,$$
where $V$ is a vector space with a basis $\{v_i\,|\,\mZ_{>0}\}$ and the first morphism corresponds to $N_i\hookrightarrow M\otimes (kv_i)$. Since $\Hom_{\bC}(M,N)=0$, applying the functors $\Hom_{\bC}(M,-)$ and $\Hom_{\bC}(M_i,-)$ yields, for each $i>0$, a commutative diagram with exact rows
$$\xymatrix{
0\ar[r]& \Hom_k(M,V)\ar@{->>}[d]\ar[r]&\Hom_k(M,V)\ar[r]\ar@{->>}[d]&\Ext^1_{\bC}(M,N)\ar[r]\ar@{->>}[d]&0\\
0\ar[r]&K_i\ar[r]&\Hom_k(M_i,V)\ar[r]&\Ext^1_{\bC}(M_i,N)\ar[r]&0,
}$$
with $K_i$ the cokernel of $\Hom_{\bC}(M_i,N)\hookrightarrow \Hom_k(M_i,V)$.  It follows easily that $K_i$ has dimension $i(i-1)/2$.
Let $\langle \cdot,\cdot\rangle$ be the bilinear form  on $V$ for which $\{v_i\}$ is an orthonormal basis.
Then we can interpret the left vertical morphism as
$$\Hom_k(M,V)\tto K_i,\quad f\mapsto (\langle v_a, f(e_b)\rangle)_{a+b\le i}.$$

Consequently, we can interpret the map to the inverse limit as
$$\Hom_k(M,V)\to\varprojlim K_i\simeq (M\otimes V)^\ast,\quad  f\mapsto \phi_f,\qquad\mbox{with $\phi_f(x\otimes v)=\langle v, f(x)\rangle$,}
$$
which is a monomorphism but not an isomorphism.

Since the spaces $K_i$ are finite dimensional, the Mittag-Leffler condition is satisfied and we have a commutative diagram with exact rows
$$\xymatrix{
0\ar[r]& \Hom_k(M,V)\ar@{^{(}->}[d]\ar[r]&\Hom_k(M,V)\ar[r]\ar@{=}[d]&\Ext^1_{\bC}(M,N)\ar[r]\ar[d]&0\\
0\ar[r]&\varprojlim K_i\ar[r]&\Hom_k(M,V)\ar[r]&\varprojlim\Ext^1_{\bC}(M_i,N)\ar[r]&0.
}$$
Since the left downwards arrow is a monomorphism which is not an epimorphism, the right downwards arrow is a epimorphism which is not a monomorphism.
\end{proof}

For completeness we also mention examples of the fact that an abelian category $\bC$, despite being extension full in $\Ind\bC$, need not be a Serre subcategory of $\Ind\bC$ (as it need not be a topologising subcategory).
\begin{example}
\begin{enumerate}
\item  For any non-noetherian coherent ring $R$, the abelian category $\bC:=R\mbox{-Mod}_{fp}$ of finitely presented modules is an abelian but not a topologising subcategory of $\Ind\bC=R\mbox{-Mod}$, the category of all modules. For instance, if $R=k[x_1,x_2,x_3,\ldots]$, then $R$ itself is in $\bC$, but its one-dimensional quotients are not.
\item Let $\bC$ be the category of countably dimensional vector spaces over $k$. Then any infinite dimensional vector space $V$ in $\bC$ has a subobject in $\Ind\bC$, which does not belong to $\bC$, of the form $X:=\varinjlim_{i\in\mN}k^i$ (where the direct limit is taken in $\Ind\bC$, not in $\bC$). Indeed, since direct limits of short exact sequences are exact in $\Ind\bC$ by definition, $X$  is a subobject of $V$. It follows from \eqref{eqlimlim} that $X$  is not isomorphic to any object in $\bC$. Concretely, since for any $U\in\bC$, we have $\Hom(U,X)=\varinjlim\Hom(U,k^i)$, any morphism (in particular an isomorphism) $U\to X$ factors as $U\to k^j\to X$ for some $j\in\mN$. Hence if $X\in\bC$, it has to be finite dimensional. But $\Hom(X,k)=\varprojlim\Hom_k(k^i,k)$ clearly contradicts that $X$ be a finite dimensional space.
\item Assume char$(k)=0$ and take $t\in k$ transcendental over $\mathbb{Q}$ (although $t\not\in\mZ$ suffices). In \cite[\S 2.19]{Del90}, Deligne constructs a tensor category $\mathbf{T}$ which is not of finite length. One sees easily that $X_t\in\mathbf{T}$ contains a non-compact subobject in $\Ind\mathbf{T}$.
\end{enumerate}
\end{example}

\section{Lower finite highest weight categories}\label{SecDef}
Fix an algebraically closed field $k$ for the entire section.
\subsection{Comparing definitions}
Let $\bC$ denote a finite length category, see Section~\ref{SFL}.

\subsubsection{} Denote by $\Lambda$ a set labelling the isomorphism classes of simple objects in $\bC$. For $\lambda\in\Lambda$, we denote the corresponding simple object by $L(\lambda)$ and its injective hull in $\Ind\bC$ by $I(\lambda)$. By our assumptions, we have $\End_{\bC}(L(\lambda))=k$. For any subset $\Omega\subset\Lambda$, we consider the Serre subcategory $\bC_{\Omega}$ of $\bC$ generated by the simple objects $L(\lambda)$ with $\lambda\in\Omega$. We denote by $\Gamma_{\Omega}:\bC\to\bC_{\Omega}$ the right adjoint of the inclusion functor, and use the same notation for the corresponding functor $\Ind\bC\to\Ind\bC_\Omega$. Clearly, if $\Omega$ is a finite set, the abelian category $\bC_\Omega$ is finite if and only if $\Gamma_\Omega I(\lambda) \in\bC$ (and hence $\Gamma_\Omega I(\lambda) \in\bC_\Omega$), for all $\lambda\in\Omega$.

 The notation $(\bC,\le)$ will stand for $\bC$ considered together with a partial order $\le$ on $\Lambda$. A partial order on $\Lambda$ is {\bf lower finite} if the set
 $$\le\lambda\;:=\;\{\mu\in\Lambda\,|\, \mu\le \lambda\}$$ is finite, for each $\lambda\in\Lambda$. We also use notation as $<\lambda$ and $\ge \lambda$ for similarly defined sets. A subset of $\Lambda$ is an {\bf ideal} if it is the union of subsets $\le\lambda$.

\begin{definition}\label{DefHWC}
For a lower finite partial order $\le$ on $\Lambda$, the pair $(\bC,\le)$ is a {\bf lower finite highest weight (lfhw) category} if, for a fixed $\diamond\in\{\ast,\dagger\}$, we have that for each $\lambda\in\Lambda$  (i$\diamond$)-(iii$\diamond$) below is satisfied.
\begin{enumerate}
\item[(i$\ast$)]  The category $\bC_{\le\lambda}$ is finite abelian.
\item[(ii$\ast$)] We have $[\nabla(\lambda):L(\lambda)]=1$, for $\nabla(\lambda):=\Gamma_{\le\lambda}I(\lambda)\in\bC_{\le\lambda}$.
\item[(iii$\ast$)] The object $I(\lambda)$ has a directed set of subobjects $I_\alpha\subset I(\lambda)$, with $I(\lambda)=\cup_\alpha I_\alpha\simeq \varinjlim_\alpha I_\alpha$, such that each $I_\alpha$ has a finite filtration with each quotient isomorphic to $\nabla(\mu)$ for $\mu\ge \lambda$ (in particular $I_\alpha\in\bC$).
\end{enumerate}

\begin{enumerate}
\item[(i$\dagger$)]  The object $L(\lambda)$ admits an injective hull $\nabla(\lambda)$ and projective cover $\Delta(\lambda)$ in $\bC_{\le\lambda}$.
\item[(ii$\dagger$)] We have $[\nabla(\lambda):L(\lambda)]=1$ and $\nabla(\lambda)$ remains injective as an object in $\bC_{\Omega}$, for each finite ideal $\Omega\ni \lambda$ in which $\lambda$ is maximal.
\item[(iii$\dagger$)] We have $\Ext^2_{\bC}(\Delta(\lambda),\nabla(\mu))=0$, for each $\mu\in\Lambda$.
\end{enumerate}
\end{definition}

The $\ast$-definition is a direct reformulation of \cite[Definition~3.53]{BS}. The $\dagger$-definition is a minor (equivalent) simplification of \cite[Definition~2.1]{RW}, see Remark~\ref{RemRW}. The following theorem is a generalisation of part of \cite[Theorem~1]{DR}.

\begin{theorem}\label{ThmDef}
The $\ast$-version and $\dagger$-version of Definition~\ref{DefHWC} are equivalent.
\end{theorem}
The proof of the theorem will occupy the remainder of this subsection.

 \begin{remark}
\label{RemQ}
\begin{enumerate}
\item Under assumption~\ref{DefHWC}(i$\ast$), the simple object $L(\lambda)$ admits a projective cover in $\bC_{\le\lambda}$, which we denote by $\Delta(\lambda)$, in accordance with \ref{DefHWC}(i$\dagger$). The objects $\Delta(\lambda)$ and $\nabla(\lambda)$ are the {\bf standard} and {\bf costandard} objects of $(\bC,\le)$.
\item By a {\bf$\nabla$-flag} of an object in $\bC$ we mean a (finite) filtration with each quotient given by a costandard object.

\item Under condition \ref{DefHWC}(iii$\ast$), the quotient $Q(\lambda):=I(\lambda)/\nabla(\lambda)$ satisfies the exact same property as attributed to $I(\lambda)$. Indeed, since $I_\alpha$ (when non-zero) has simple socle $L(\lambda)$ and a $\nabla$-flag, it must contain $\nabla(\lambda)$ in a way that the quotient $I_\alpha/\nabla(\lambda)$ inherits each $\nabla$-flag from $I_\alpha$.
\item For $\diamond\in\{\ast,\dagger\}$ and $a\in\{i,ii,iii\}$, we will abbreviate `assume that \ref{DefHWC}(a$\diamond$) is satisfied for every $\lambda\in\Lambda$', to `assume (a$\diamond$)'.
\end{enumerate}
\end{remark}

\begin{lemma}
\label{Lemab}
Assume either (i$\ast$)-(iii$\ast$) or (i$\dagger$)-(ii$\dagger$). For all $\lambda,\mu\in\Lambda$
\begin{enumerate}
\item $\dim_k\Hom(\Delta(\lambda),\nabla(\mu))=\delta_{\lambda,\mu}$;
\item $\Ext^1_{\bC}(\Delta(\lambda),\nabla(\mu))=0$.
\end{enumerate} \end{lemma}
\begin{proof}
Claim (1) is immediate from the definition of $\Delta(\lambda)$ and $\nabla(\mu)$.

If $\mu\le\lambda$, then the extension group in (2) can be calculated in $\bC_{\le\lambda}$, where it is zero since $\Delta(\lambda)$ is by definition projective in $\bC_{\le\lambda}$. So we consider the case $\mu\not\le\lambda$.

First we use assumption (iii$\ast$). Denote by $Q(\mu)$ the cokernel of $\nabla(\mu)\hookrightarrow I(\mu)$. Then $\Ext^1_{\bC}(\Delta(\lambda),\nabla(\mu))$ is a quotient of $\Hom(\Delta(\lambda),Q(\mu))$. The latter space is zero by $\mu\not\le\lambda$, Remark~\ref{RemQ}(3) and claim (1).

Now we use assumption (ii$\dagger$). Denote by $\Omega$ the ideal generated by $\lambda$ and $\mu$. The extension group can be calculated in $\bC_{\Omega}$. Since $\mu\not\le\lambda$, it follows that $\mu$ is maximal in $\Omega$ and hence $\nabla(\mu)$ is injective in $\bC_{\Omega}$, so the extension group is zero.
\end{proof}

\begin{lemma}\label{Lemb}
Assume (i$\dagger$)-(ii$\dagger$). For each $\lambda,\mu\in\Lambda$ we have:
\begin{enumerate}
\item $[\Delta(\lambda):L(\lambda)]=1$;
\item $\Ext^1_{\bC}(\nabla(\lambda),\nabla(\mu))=0$ unless $\lambda>\mu$;
\item $\Delta(\lambda)$ remains projective in $\bC_{\Omega}$, for each finite ideal $\Omega$ in which $\lambda$ is maximal.
\end{enumerate}
\end{lemma}
\begin{proof}
The definition of $\nabla(\lambda)$, respectively $\Delta(\lambda)$, implies
$$[\Delta(\lambda):L(\lambda)]\;=\;\dim_k \Hom_{\bC_{\le\lambda}}(\Delta(\lambda),\nabla(\lambda))\;=\; [\nabla(\lambda):L(\lambda)],$$
which demonstrates claim (1). 
Claim (2) follows from \ref{DefHWC}(ii$\dagger$). 

If claim (3) were not true, then there exists $\nu$, with $\nu\not\le\lambda$ and $\nu\not\ge\lambda$ with
$$\Ext^1(\Delta(\lambda),L(\nu))\not=0.$$
If we denote the cokernel of $L(\nu)\hookrightarrow \nabla(\nu)$ by $Q$ we have an exact sequence
$$\Hom(\Delta(\lambda),Q)\;\to\;\Ext^1(\Delta(\lambda),L(\nu))\;\to\; \Ext^1(\Delta(\lambda),\nabla(\nu)).$$
By Lemma~\ref{Lemab}(2), the right term is zero. The left term is zero since $\lambda\not\le\nu$. This concludes the proof.
\end{proof}

\begin{lemma}[Humphreys-BGG reciprocity]\label{LemPf}
Assume (i$\dagger$)-(iii$\dagger$) and fix $\lambda\in\Lambda$. Define 
$$I_\Omega:=\Gamma_\Omega I(\lambda)\subset I(\lambda),$$ for each finite ideal $\Omega\subset\Lambda$.
For any finite ideal $\Omega'$ in which $\mu$ is maximal and $\Omega:=\Omega'\backslash\{\mu\}$, we have
$$I_{\Omega'}/I_\Omega\;\simeq\; \nabla(\mu)^{n_\mu},\quad\mbox{with }\;\,n_\mu:=[\Delta(\mu):L(\lambda)].$$
\end{lemma}
\begin{proof}
We prove the statement in the lemma for triples $(\Omega',\Omega,\mu)$ of finite ideals $\Omega',\Omega$ with $\{\mu\}=\Omega'\backslash\Omega$, by induction along the inclusion order for the sets $\Omega'$. We take such a minimal $(\Omega',\Omega,\mu)$ for which the statement is not yet known. We set $Q:=I_{\Omega'}/I_\Omega$. 
By construction, we have $\Hom(L(\kappa),Q)=0$, for any $\kappa\not=\mu$ in $\Lambda$, from which it also follows that
$$\Hom(L(\mu),Q)\,\simeq\, \Hom(\Delta(\mu),Q).$$
Moreover, we have an exact sequence
$$0\to\Hom(\Delta(\mu),I_{\Omega'})\to \Hom(\Delta(\mu),Q)\to\Ext^1(\Delta(\mu),I_\Omega).$$
By our minimality assumption on $\Omega'$, we know that $I_\Omega$ has a $\nabla$-flag, so by Lemma~\ref{Lemab}(2) the extension group on the right is zero. Combining the statements in this paragraph, and the definition of $I_{\Omega'}$ as injective hull of $L(\lambda)$ in $\bC_{\Omega'}$, thus yields
$$\dim\Hom(L(\kappa),Q)=\begin{cases}
0&\mbox{if $\kappa\not=\mu$}\\
[\Delta(\mu):L(\lambda)]=n_\mu&\mbox{if $\kappa=\mu$.}\end{cases}$$

Therefore, we can embed $Q$ into the injective object $\nabla(\mu)^{n_\mu}$ in $\bC_{\Omega'}$ and we denote the cokernel by $Q'$. Furthermore, for any $\nu\in\Lambda$, we have the exact sequence
$$\Ext^1_{\bC}(\Delta(\nu),I_{\Omega'})\to\Ext^1_{\bC}(\Delta(\nu),Q)\to\Ext^2_{\bC}(\Delta(\nu), I_\Omega).$$
Since $I_\Omega$ has a $\nabla$-flag, assumption (iii$\dagger$) shows that the right term is zero.  If $\nu\in\Omega'$, also the left term is zero by definition of $I_{\Omega'}$. Hence, for every $\nu\in\Omega'$ we find $\Ext^1_{\bC}(\Delta(\nu),Q)=0$.  Using this, we find a short exact sequence
$$0\to\Hom(\Delta(\nu),Q)\to \Hom(\Delta(\nu), \nabla(\mu)^{n_\mu})\to\Hom(\Delta(\nu), Q')\to0,\quad\mbox{for every $\nu\in\Omega'$}.$$
By combining this with Lemma~\ref{Lemab}(1) we find $\Hom(\Delta(\nu), Q')=0$, for every $\nu\in\Omega'$ which demonstrates $Q'=0$, or $I_{\Omega'}/I_{\Omega}\simeq\nabla(\mu)^{n_\mu}$ as desired.
\end{proof}

\begin{proof}[Proof of Theorem~\ref{ThmDef}]
First we assume that conditions \ref{DefHWC}(i$\ast$)-(iii$\ast$) are satisfied. Then clearly \ref{DefHWC}(i$\dagger$) and the multiplicity 1 condition in \ref{DefHWC}(ii$\dagger$) are also satisfied.
Denote the quotient of $\nabla(\mu)\hookrightarrow I(\mu)$ by $Q(\mu)$. We then have an exact sequence
$$\Ext^1(\Delta(\lambda),Q(\mu))\;\to\;\Ext^2(\Delta(\lambda),\nabla(\mu))\;\to\;\Ext^2(\Delta(\lambda),I(\mu)).$$
Clearly the right term is zero. The left term is zero by Remark~\ref{RemQ}(3), Theorem~\ref{ThmInd}(2) and Lemma~\ref{Lemab}(2). Hence the extension group in the middle also vanishes and \ref{DefHWC}(iii$\dagger$) is satisfied.
To prove the injective property of $\nabla(\lambda)$ in (ii$\dagger$), it suffices to show that 
$$\Ext^1(L(\nu),\nabla(\lambda))=0, \quad\mbox{for $\nu\not\ge\lambda$ and $\nu\not\le\lambda$.}$$
This is indeed true, since the extension group is a quotient of $\Hom(L(\nu),Q(\lambda))$ which is zero itself by Remark~\ref{RemQ}(3). Hence \ref{DefHWC}(i$\dagger$)-(iii$\dagger$) are also satisfied.

Now assume that conditions \ref{DefHWC}(i$\dagger$)-(iii$\dagger$) are satisfied. Clearly also \ref{DefHWC}(ii$\ast$) is satisfied.
Let $(S, \subseteq)$ be the directed set of all finite ideals in $\Lambda$. Then we have
$$I(\lambda)\;\simeq\;\varinjlim_{\Omega \in S}I_\Omega,\qquad\mbox{with}\quad I_\Omega:=\Gamma_\Omega I(\lambda)\subset I(\lambda).$$
Lemma~\ref{LemPf} shows that \ref{DefHWC}(iii$\ast$) is satisfied.

To show that \ref{DefHWC}(i$\ast$) is satisfied, for any $\mu\le\lambda$ we need that $\Gamma_{\le \lambda}I(\mu)$ is contained in $\bC$. We have already demonstrated the stronger claim that $\Gamma_{\le \lambda}I(\mu)$ has a $\nabla$-flag.
\end{proof}

\begin{remark}\label{RemRW}
For $(\bC,\le)$ to be a `highest weight category' in \cite[Definition~2.1]{RW}, \ref{DefHWC}(i$\dagger$)-(iii$\dagger$) need to be satisfied, along with conditions \ref{Lemb}(1) and (3). These additional assumptions are thus redundant and \cite[Definition~2.1]{RW} is equivalent to Definition~\ref{DefHWC}.
\end{remark}

The following technical lemma will be useful later on. 

\begin{lemma}\label{LemBGG}
For a lfhw category $(\bC,\le)$ with $\mu\in\Lambda$ and a finite subset (not necessarily an ideal) $\Sigma\subset\Lambda$ which satisfies the properties
\begin{itemize}
\item $\Sigma$ contains $\le\mu$, and
\item $\Sigma\cap\{\nu\in\Lambda\,|\, [\nabla(\nu):L(\mu)]\not=0\}\;=\;\{\mu\}$,
\end{itemize}
we have
$$[\Delta(\mu):L(\lambda)]\;=\;[\Gamma_{\Sigma}I(\lambda):L(\mu)],\quad\mbox{for all $\lambda\in\Lambda$}.$$
\end{lemma}
\begin{proof}
We fix $\lambda\in\Lambda$ for the entire proof. By Lemma~\ref{LemPf} we have
 $$[\Delta(\mu):L(\lambda)]\,=\,[\Gamma_{\le \mu}I(\lambda):L(\mu)].$$
We denote the quotient of $\Gamma_{\le\mu}I(\lambda)\hookrightarrow \Gamma_{\Sigma}I(\lambda)$ by $Q$. Since $\Gamma_{\Sigma}$ is left exact and $\Gamma_{\Sigma}\nabla(\kappa)=0$ for $\kappa\not\in\Sigma$, Lemma~\ref{LemPf} implies that
$$[Q:L(\mu)]\,\le\,\sum_{\nu\in\Sigma\backslash\{\mu\}}n_{\nu}[\Gamma_{\Sigma}\nabla(\nu):L(\mu)]\,\le\,\sum_{\nu\in\Sigma\backslash\{\mu\}}n_{\nu}[\nabla(\nu):L(\mu)]\;=\;0.$$
This shows that $[\Gamma_{\le\mu}I(\lambda):L(\mu)]=[\Gamma_{\Sigma}I(\lambda):L(\mu)]$ and hence concludes the proof.
\end{proof}

\subsection{Serre subcategories and quotients} Fix a lfhw category $(\bC,\le)$.
We review some well-known homological properties of Serre subcategories and quotient categories of $\bC$. 

\begin{lemma}\label{LemExt}
Let $\Omega$ be a finite ideal in $\Lambda$ and $\mu\in\Lambda\backslash \Omega$. For any $M\in\bC_{\Omega}$ we have
$$\Ext^i_{\bC}(M,\nabla(\mu))\;=\;0,\qquad\mbox{for all $i\in\mN$}.$$
\end{lemma}
\begin{proof}
The claim for $i=0$ is obvious. Now consider $i>0$ and assume we already know the claim for $i-1$ (and for all $\mu\not\in\Omega$). For $Q(\mu)$ the quotient of $\nabla(\mu)\hookrightarrow I(\mu)$, there is an epimorphism
$$\Ext^{i-1}(M,Q(\mu))\;\tto\; \Ext^i(M,\nabla(\mu)).$$
By Remark~\ref{RemQ}(3), $Q(\mu)$ is a direct limit of objects which have a finite filtration with all quotients of the form $\nabla(\nu)$ for $\nu\not\in\Omega$.
The left-hand is therefore zero by the induction hypothesis and Theorem~\ref{ThmInd}(2). This concludes the proof.
\end{proof}

\begin{corollary}\label{CorOmegaEF}
Let $\Omega$ be a finite ideal in $\Lambda$, then $\bC_{\Omega}$ is extension full in $\bC$.
\end{corollary}
\begin{proof}
By Theorem~\ref{ThmInd}(1), we can instead prove that $\bC_\Omega$ is extension full in $\Ind\bC$.
The category $\bC_{\Omega}$ has enough injective objects, since Lemma~\ref{LemPf} implies that $\Gamma_{\Omega}I(\mu)$ is in $\bC$ for each $\mu\in\Omega$. Furthermore, every injective object in $\bC_\Omega$ is a (finite) direct sum of injective envelopes of simple objects.
By \cite[Proposition~4]{Sigma}, it therefore suffices to show that $\Ext^i_{\bC}(M,\Gamma_{\Omega}I(\mu))=0$ for all $i>0$, $M\in\bC_\Omega$ and $\mu\in\Omega$. For each $i>0$ we have an epimorphism
$$\Ext^{i-1}_{\bC}(M,I(\mu)/\Gamma_{\Omega}I(\mu))\;\tto\;\Ext^i_{\bC}(M,\Gamma_{\Omega}I(\mu)).$$
The left term is zero, by application of Lemmata~\ref{LemPf} and \ref{LemExt}.
\end{proof}

\subsubsection{} For a finite ideal $\Omega\subset\Lambda$, we denote the Serre quotient $\bC/\bC_{\Omega}$ by $\bC^{\Omega}$. For example $\bC^\varnothing\simeq\bC$. By construction, the simple objects in $\bC^{\Omega}$ are labelled by $\Lambda\backslash \Omega$ and we denote the (lower finite) partial order on $\Lambda\backslash \Omega$ obtained by restriction of $\le$ again by $\le$.

\begin{lemma}\label{LemQuo}
The category $(\bC^\Omega,\le)$ is a lfhw category. For any $\lambda\in\Lambda\backslash \Omega$, the (co)standard objects are $\pi\Delta(\lambda)$ and $\pi\nabla(\lambda)$.\end{lemma}
\begin{proof}
We have a bijection between lower finite ideals in $\Lambda\backslash \Omega$ and lower finite ideals in $\Lambda$ containing $\Omega$.
For such an ideal $\Pi\supset \Omega$, we denote by $\bC^\Omega_{\Pi}$ the corresponding Serre subcategory of $\bC^\Omega$, or equivalently the relevant Serre quotient of $\bC_\Pi$.

We will prove that the $\dagger$-version of Definition~\ref{DefHWC} is applicable to $\bC^\Omega$. For any Serre subcategory $\bB\subset\bA$, the exact functor $\pi:\bA\to\bA/\bB$ sends a projective cover (resp. injective hull) of a simple object not contained in $\bB$ to the corresponding projective cover (resp. injective hull) in $\bB/\bA$.
 Conditions (i$\dagger$)-(ii$\dagger$) for $\bC^\Omega$ then follow. Next assume that 
 $$\Ext^2_{\bC^{\Omega}}(\Delta(\lambda),\nabla(\mu))\,\not=\,0,\quad\mbox{for some $\lambda,\mu\in\Lambda\backslash \Omega$.}$$
 Since we can compute the extension group in $\bC^\Omega$, rather than $\Ind\bC^\Omega$, it follows immediately that there must be a non-trivial extension which exists in $\bC^\Omega_{\Pi}$ for a finite ideal $\Pi$ containing $\Omega$. 
 However, based on Lemma~\ref{LemPf} the cokernel of the embedding of $\nabla(\mu)$ into its injective envelope in $\bC_{\Pi}$ has a $\nabla$-flag. The same therefore follows for $\nabla(\mu)$ in $\bC_{\Pi}^\Omega$. However, we have 
 $$\Ext^1_{\bC^{\Omega}}(\Delta(\lambda),\nabla(\nu))\,=\,0,\quad\mbox{for all $\lambda,\nu\in\Lambda$,}$$
 by Lemma~\ref{Lemab}(2), a contradiction.
 \end{proof}

\section{Uniqueness of lower finite highest weight structures}
\label{SecUniq}
\subsection{The essential order}

\begin{definition}

Let $\bC$ be a finite length category with lower finite partial orders $\le$ and $\le'$ on $\Lambda$ such that $(\bC,\le)$ and $(\bC,\le')$ are lfhw categories. The two highest weight structures are {\bf equivalent} if the costandard objects in $(\bC,\le)$ and $(\bC,\le')$ are isomorphic.
\end{definition}

\begin{lemma}\label{Lempext}
For a lfhw category $(\bC,\le)$, let  $\le'$ be a lower finite extension of the partial order $\le$ on $\Lambda$. Then $(\bC,\le')$ is a lfhw category and the highest weight structures $(\bC,\le)$ and $(\bC,\le')$ are equivalent.
\end{lemma}
\begin{proof}
We use the $\dagger$-version of Definition~\ref{DefHWC}. Since $\le'$ extends $\le$ and is lower finite, the set $\Omega$ defined as $\le'\lambda$ is a finite ideal in $(\Lambda,\le)$ and $\lambda$ is maximal in $\Omega$ under $\le$. Condition \ref{DefHWC}(ii$\dagger$) and Lemma~\ref{Lemb}(3) demonstrate that $\Delta(\lambda)$ and $\nabla(\lambda)$ are the projective cover and injective hull of $L(\lambda)$ in $\bC_{\le'\lambda}$. Since every finite ideal in $(\Lambda,\le')$ is in particular a finite ideal in $(\Lambda,\le)$ condition \ref{DefHWC}(ii$\dagger$) for $(\bC,\le')$ follows from the fact it holds for $(\bC,\le)$. Finally, condition~\ref{DefHWC}(iii$\dagger$) is immediately inherited. Hence $(\bC,\le')$ is a lfhw category and the equivalence claim is immediate.
\end{proof}

The lemma motivates the following definition, see \cite[Definition~1.2.5]{blocks}.

\begin{definition}\label{Defess}
For a lfhw category $(\bC,\le)$, the {\bf essential order} $\le^e$ is the partial order on $\Lambda$ generated by the two relations
$$[\Delta(\lambda):L(\mu)]\not=0\;\Rightarrow\; \mu\le^e\lambda\quad\mbox{and}\quad [\nabla(\lambda):L(\mu)]\not=0\;\Rightarrow \;\mu\le^e\lambda.$$
\end{definition}
Clearly, $\le$ is an extension of $\le^e$, so in particular, $\le^e$ is lower finite.

\begin{lemma}\label{Lemeo}
For a lfhw category $(\bC,\le)$, the pair $(\bC,\le^e)$ is a lfhw category, equivalent to $(\bC,\le)$.
\end{lemma}
\begin{proof}
We use the $\dagger$-version of Definition~\ref{DefHWC}. Since $\Delta(\lambda)$ and $\nabla(\lambda)$ belong to the Serre subcategory $\bC_{\le^e\lambda}$ of $\bC_{\le\lambda}$ they are the projective cover and injective hull of $L(\lambda)$ in $\bC_{\le^e\lambda}$. To prove that $(\bC,\le^e)$ is a lfhw category equivalent to $(\bC,\le)$ it suffices to prove that $\nabla(\lambda)$ is injective in every $\bC_{\Omega}$, for $\Omega$ a finite ideal in $(\Lambda,\le^e)$. For any $\lambda,\mu\in\Lambda$, the short exact sequence $K(\mu)\hookrightarrow\Delta(\mu)\tto L(\mu)$ yields an exact sequence
$$\Hom(K(\mu),\nabla(\lambda))\to\Ext^1(L(\mu),\nabla(\lambda))\to\Ext^1(\Delta(\mu),\nabla(\lambda)).$$
 Lemma~\ref{Lemab}(2) implies the right term is zero. The left term is zero unless $[K(\mu):L(\lambda)]\not=0$. The middle term will therefore always be zero unless $\lambda<^e\mu$. This demonstrates the requested injective property of $\nabla(\lambda)$.
 \end{proof}

We can summarise the above in the following generalisation of \cite[Lemma~1.2.6]{blocks}.
\begin{prop}\label{PropEquiv}Let $\bC$ be a finite length category. Assume that $(\bC,\le_1)$ and $(\bC,\le_2)$ are lfhw categories for lower finite partial orders $\le_1$ and $\le_2$ on $\Lambda$. The following are equivalent:
\begin{enumerate}
\item The two highest weight structures are equivalent.
\item The two essential orders $\le_1^e$ and $\le^e_2$ are the same.
\item We have $[\nabla_1(\lambda)]=[\nabla_2(\lambda)]$ in $K_0(\bC)$ for all $\lambda\in\Lambda$.
\end{enumerate}
\end{prop}
\begin{proof}
Clearly (1) implies (3). That (2) implies (1) follows from Lemma~\ref{Lemeo}.

Now we prove that (3) implies (2). We thus assume that $[\nabla_1(\nu)]=[\nabla_2(\nu)]$ for all $\nu\in\Lambda$. For a fixed $\mu\in\Lambda$, let $\Sigma$ denote the finite set which is the union of $\le_1\mu$ and $\le_2\mu$. It follows from Lemma~\ref{LemBGG} that
$$[\Delta_1(\mu):L(\lambda)]\;=\; [\Gamma_\Sigma I(\lambda):L(\mu)]\;=\; [\Delta_2(\mu):L(\lambda)].$$
The conclusion now follows from Definition \ref{Defess}.\end{proof}

\begin{lemma}\label{LemLength}
For a lfhw category $(\bC,\le)$, let $\ell:\Lambda\to\mN$ be the function
$$\ell(\lambda)\,=\,\max\{n\in\mN \,|\, \mbox{there exist $\mu_i\in\Lambda$ for $1\le i\le n$ with $\mu_n<^e\mu_{n-1}<^e\cdots<^e\mu_1<^e\lambda$}\}.$$
Set $\Lambda_d:=\ell^{-1}([0,d])$, for $d\in\mN$ (so $\Lambda_{d-1}\subset\Lambda_d$ and $\Lambda=\cup_d\Lambda_d$). Then $\nabla(\lambda)=\Gamma_{\Lambda_{\ell(\lambda)}}I(\lambda)$.
\end{lemma}
\begin{proof}
We can work in the equivalent highest weight category $(\bC,\le^e)$. It is clear that $\nabla(\lambda)\subset \Gamma_{\Lambda_{\ell(\lambda)}}I(\lambda)$.
If the inclusion is strict, there should exist $\mu\in\Lambda$ with $\ell(\mu)\le \ell(\lambda)$ such that $\Ext^1(L(\mu),\nabla(\lambda))$ does not vanish. By (ii$\dagger$) the latter requires $\mu>^e\lambda$ which means $\ell(\mu)>\ell(\lambda)$, a contradiction.
\end{proof}

\subsection{Uniqueness of essential order}
\begin{theorem}\label{ThmUniq}
Let $\bC$ be a finite length category which admits a simple preserving anti-autoequivalence. Up to equivalence, $\bC$ admits at most one highest weight structure.
\end{theorem}

The rest of the subsection is devoted to the proof.

\begin{prop}\label{PropIdeal}
If $(\bC,\le)$ is a lfhw category with simple preserving anti-autoequivalence, then
a finite subset $\Omega\subset\Lambda $ is an ideal in $(\Lambda,\le^e)$ if and only if we can order the elements in $\Omega$ as $\{\mu_1,\mu_2,\cdots,\mu_n\}$ such that
$$\Ext^2_{\bC^{\Omega_i}}(L(\mu_i),L(\mu_i))=0,\quad\mbox{for $1\le i\le n$},$$
with $\Omega_i=\{\mu_1,\mu_2,\cdots,\mu_{i-1}\}\subset\Omega$.
\end{prop}
We start by proving the case $|\Omega|=1$ of the proposition.
\begin{lemma}\label{LemExt2}
Consider a lfhw category $(\bD,\preceq)$ with a simple preserving anti-autoequivalence $\Psi$.
The following are equivalent for $\lambda\in\Lambda$:
\begin{enumerate}
\item $\lambda$ is minimal under $\preceq^e$ (that is $\ell(\lambda)=0$ or $\Delta(\lambda)=L(\lambda)=\nabla(\lambda)$);
\item $\Ext^2_{\bD}(L(\lambda),L(\lambda))=0$;
\item $\Ext^\bullet_{\bD}(L(\lambda),L(\lambda))=k$.
\end{enumerate}
\end{lemma}
\begin{proof}
Claim (3) implies claim (2) as a special case. If (1) is satisfied, then Corollary~\ref{CorOmegaEF} applied to the lfhw category $(\bD,\preceq^e)$ implies 
$$\Ext^\bullet_{\bD}(L(\lambda),L(\lambda))\;\simeq\; \Ext^\bullet_{\bD_{\{\lambda\}}}(L(\lambda),L(\lambda)).$$
By Lemma~\ref{Lemb}(1), the projective cover $\Delta(\lambda)$ of $L(\lambda)$ in $\bD_{\{\lambda\}}$ is simple and hence $\bD_{\{\lambda\}}$ is semisimple. Therefore (1) implies (3).

Finally we show that (2) implies (1).
By Corollary~\ref{CorOmegaEF}, we can compute the extension group of (2) in $\bD_{\preceq\lambda}$. Denote the kernel of $\Delta(\lambda)\to L(\lambda)$ by $K(\lambda)$ and the cokernel of $L(\lambda)\hookrightarrow \nabla(\lambda)$ by $C(\lambda)$.
Applying the fact that $\Delta(\lambda)$, resp. $\nabla(\lambda)$, is projective, resp. injective, in $\bD_{\preceq \lambda}$ shows that
\begin{equation}\label{eqExt2}0=\Ext^2_{\bD_{\preceq\lambda}}(L(\lambda),L(\lambda))\simeq \Ext^1(K(\lambda),L(\lambda))\simeq \Hom(K(\lambda),C(\lambda)).\end{equation}
By Definition \ref{DefHWC}(i$\dagger$) we have $\nabla(\lambda)\simeq\Psi(\Delta(\lambda))$. Hence the top of $K(\lambda)$ is isomorphic to the socle of $C(\lambda)\simeq \Psi(K(\lambda))$, and the morphism space in \eqref{eqExt2} can only be zero if $K(\lambda)=0=C(\lambda)$.        
\end{proof}

\begin{remark}
More generally, the technique in the proof of Lemma~\ref{LemExt2} shows that for any subobject $D\subset \nabla(\lambda)$, we have
$\Ext^2_{\bC}(\Psi D,D)=0$ if and only if $D=\nabla(\lambda).$
\end{remark}

\begin{proof}[Proof of Proposition~\ref{PropIdeal}]
Assume first that $\Omega$ is an ideal in $\le^e$. Then we can order the elements in $\Omega$ as $\{\mu_1,\ldots,\mu_n\}$, such that each $\Omega_i$, as defined in the proposition, is also an ideal. Then $\mu_i$ is minimal in $\Lambda\backslash\Omega_i$, for $1\le i\le n$. Hence Lemma~\ref{LemExt2}, applied to the lfhw category $(\bC^{\Omega_i},\le^e)$ implies all the relevant extensions vanish. 

Now assume that we can order the set $\Omega$ such that the extensions in the proposition vanish. We prove by induction on $n=|\Omega|$ that $\Omega$ is an ideal in $(\Lambda,\le^e)$. If $n=1$, then the claim follows from Lemma~\ref{LemExt2}. Next assume that we know the claim is true for $n-1$. Then by assumption we already know that $\Omega_{n}$ is an ideal. Then $(\bC^{\Omega_n},\le^e)$ is a lfhw category by Lemma~\ref{LemQuo}, so Lemma~\ref{LemExt2} shows that $\mu_n$ is minimal in $\Lambda\backslash\Omega_n$ under $\le^e$. Thus $\Omega$ is indeed also an ideal in $(\Lambda,\le^e)$.
\end{proof}

\begin{proof}[Proof of Theorem~\ref{ThmUniq}]
By Proposition~\ref{PropEquiv}, it suffices to prove that every lfhw structure leads to the same essential order. Since any lower finite partial order is determined by its set of finite ideals, the conclusion follows from Proposition~\ref{PropIdeal}.
\end{proof}

\begin{remark}
In \cite{CPS2}, certain highest weight categories with a `duality' (meaning simple preserving involutive contravariant functor) are studied. Theorem~\ref{ThmUniq} demonstrates that under those circumstances the category is highest weight for at most one lower finite partial order (up to equivalence).
\end{remark}

\begin{remark}
We observe that the condition in Lemma~\ref{LemExt2}
$$\dim_k\Ext^\bullet (L(\lambda),L(\lambda))=1$$
for a simple object to be a standard object in a {\em lower finite} highest weight category with duality, is `Koszul dual' to the condition (as derived in \cite{blocks})
$$\dim_k\Hom(P(\lambda),P(\lambda))=1$$
for a projective object to be a standard object in an {\em upper finite} highest weight category with duality.
\end{remark}

\subsection{Remarks on Ringel duality}

Let $(\bC,\le)$ be a lfhw category. 
\subsubsection{}
By \cite[Theorem~4.2]{BS}, for every $\lambda\in\Lambda$ there exists a unique indecomposable tilting object $T(\lambda)\in\bC_{\le\lambda}$ with $[T(\lambda):L(\lambda)]\not=0$. Recall that tilting objects are the objects in $\bC$ which admit both a $\nabla$-flag and a $\Delta$-flag.
Moreover, by \cite[4.19-4.20]{BS}, there is an appropriate category of modules over the category of tilting objects, which forms an upper finite highest weight category, the Ringel dual of  $(\bC,\le)$. It is proved in \cite[\S 4.3]{BS} that upper finite highest weight categories also have a theory of tilting objects, and in \cite[Corollary~4.24]{BS} that taking the Ringel dual of the Ringel dual returns the original lfhw category (up to equivalence).

Any simple preserving anti-autoequivalence on $\bC$ restricts to an anti-autoequivalence on the category of tilting objects in $(\bC,\le)$, and every indecomposable tilting object will be preserved. It then follows from \cite[4.3.1]{blocks} that the Ringel dual of $(\bC,\le)$ admits no second (non-equivalent) highest weight structure. From the above Ringel duality we can then deduce (without using Theorem~\ref{ThmUniq}) that $\bC$ cannot admit a second highest weight structure {\em with the same category of tilting objects}. Perhaps this observation can be used to find an alternative proof of Theorem~\ref{ThmUniq}.

\subsubsection{}
 We define the tilting order as the partial order $\le^t$ on $\Lambda$ generated by the two relations
$$(T(\lambda):\Delta(\mu))\not=0\;\Rightarrow\; \mu\le^t\lambda\quad\mbox{and}\quad (T(\lambda):\nabla(\mu))\not=0\;\Rightarrow \;\mu\le^t\lambda.$$
By Lemma~\ref{Lemeo}, we have $T(\lambda)\in\bC_{\le^e\lambda}$, so $\mu\le^t\lambda$ implies that $\mu\le^e\lambda$.

It follows from the equivalences of categories in \cite[4.20-4.22]{BS} and standard reciprocity laws, such as Lemma~\ref{LemPf}, that the essential order on the Ringel dual of $(\bC,\le)$ is the reversal of $\le^t$ and the tilting order is the reversal of $\le^e$ (both orders are defined by the same formulae as in the lower finite case). This interpretation of the partial order shows that $\mu\le^e\lambda$ implies $\mu\le^t\lambda$. In conclusion, the tilting and essential order coincide, for any given lfhw category.

\subsection{Connection with the CPS setting}\label{SecCPS}
For a finite length category $\bC$, a priori the category $\Ind\bC$ could also be a highest weight category in the sense of \cite[Definition~3.1]{CPS}. In this subsection, we demonstrate that we can extend Theorem~\ref{ThmUniq} by showing that, under the same assumptions, $\Ind\bC$ cannot be a highest weight category in any other way than via its (unique) lfhw structure.

\subsubsection{} We let $(\bA,\sqsubseteq)$ be a highest weight category in the sense of  \cite[Definition~3.1]{CPS}. In particular $\bA$ is `locally artinian' and we label simple objects by a poset $(\Lambda,\sqsubseteq)$ which is only assumed interval finite. To stress that we use \cite[Definition~3.1]{CPS} and do not work with a lfhw category, we use the notation $L(\lambda)\subset A(\lambda)\subset I(\lambda)$ of \cite{CPS}, and not $\nabla(\lambda)$. Similarly we use the term `good filtration' as in \cite[3.1(b)]{CPS} and not `$\nabla$-flag'.

\begin{lemma}\label{Lem4CPS}
Fix $\lambda\in\Lambda$ and assume
$$\dim\Ext^1_{\bA}(L(\lambda),L(\mu))\,=\,\dim \Ext^1_{\bA}(L(\mu),L(\lambda)),\quad\mbox{for all $\mu\in\Lambda$}.$$
\begin{enumerate}
\item If $\Ext^2_{\bA}(L(\lambda),L(\lambda))=0$, then
$L(\lambda)=A(\lambda)$.
\item If $A(\lambda)=L(\lambda)$ then $A(\lambda)$ appears in a good filtration of $I(\nu)$ if and only if $\nu=\lambda$.
\end{enumerate}
\end{lemma}
\begin{proof}
Set $Q:=A(\lambda)/L(\lambda)$. Under the assumption in (1) we have epimorphisms
$$\Hom_{\bA}(L(\lambda),I(\lambda)/A(\lambda))\tto \Ext^1_{\bA}(L(\lambda),A(\lambda))\tto \Ext^1_{\bA}(L(\lambda),Q).$$
The left-hand space is zero by \cite[3.1(c)(ii)]{CPS}, so we find $\Ext^1(L(\lambda),Q)=0$.

If $Q\not=0$, then the locally artinian assumption implies that we have a subobject $L(\mu)\subset Q$ for some $\mu\in\Lambda$ (different from $\lambda$) and hence a corresponding exact sequence
$$\Hom_{\bA}(L(\lambda),Q/L(\mu))\to \Ext^1_{\bA}(L(\lambda),L(\mu))\to \Ext^1_{\bA}(L(\lambda),Q).$$
The left-hand space is zero and by the above paragraph, so is the right-hand one. In conclusion and by assumption, we find
$$0=\Ext^1_{\bA}(L(\lambda),L(\mu))=\Ext^1_{\bA}(L(\mu),L(\lambda)).$$
However, the exact sequence
$$0=\Hom_{\bA}(L(\mu),A(\lambda))\to \Hom_{\bA}(L(\mu),Q)\to\Ext^1_{\bA}(L(\mu),L(\lambda))=0$$
then contradicts $L(\mu)\subset Q$. In conclusion $Q=0$ and thus $L(\lambda)=A(\lambda)$, proving part (1).

Under the assumption in part (2) and by \cite[Lemma~3.2(b)]{CPS} we have
\begin{equation}\label{1LL}0=\Ext^1_{\bA}(L(\mu),L(\lambda))=\Ext^1_{\bA}(L(\lambda),L(\mu)),\qquad\mbox{for all $\mu\sqsubseteq\lambda$}.\end{equation}
If $A(\lambda)=L(\lambda)$ appears in a good filtration of $I(\nu)$, then \cite[3.1(c) and 3.2(b)]{CPS} implies that there exists $\nu\sqsubseteq \kappa\sqsubset\lambda$ such that $\Ext^1_{\bA}(L(\lambda),A(\kappa))\not=0$. Denote by $E$ such a non-split extension. Since $E$ is a union of its finite length subobjects, we have $E'\subset E$ of finite length such that $E'\to E\to L(\lambda)$ is an epimorphism. However, since $[E':L(\mu)]\not=0$ implies that $\mu\sqsubseteq\lambda$ it follows from equation~\eqref{1LL} that  $E'\tto L(\lambda)$ splits. Consequently, $E$ splits too and we find a contradiction.
\end{proof}

\subsubsection{} Consider a lfhw category $(\bC,\le)$ with a simple preserving anti-autoequivalence $\Psi$. Assume that $(\Ind\bC,\sqsubseteq)$ is also a highest weight category in the sense of \cite{CPS}, for some partial order $\sqsubseteq$ on $\Lambda$. Existence of $\Psi$ implies we can apply Lemma~\ref{Lem4CPS} to $\bA=\Ind\bC$. In combination with Lemma~\ref{LemExt2} we find that if $\ell(\lambda)=0$, then $A(\lambda)=L(\lambda)=\nabla(\lambda)$ and $A(\lambda)$ appears in a good filtration of $I(\nu)$ if and only if $\nu=\lambda$. In particular, $(\Ind\bC,\sqsubseteq')$ is still a highest weight category for the order $\sqsubseteq'$ obtained from $\sqsubseteq$ by removing all relations $\nu\sqsubset\lambda$ with $\ell(\lambda)=0$. Let $\bB$ the Serre subcategory of objects in $\Ind\bC$ which only have non-zero multiplicities for simple objects $L(\lambda)$ with $\ell(\lambda)=0$. It follows that $(\Ind\bC)/\bB$ is again canonically a highest weight category. By iteration, we find that $(\Ind\bC,\sqsubseteq'')$ is a highest weight category (with the same $A(\mu)$) for some restriction $\sqsubseteq''$ of $\sqsubseteq$ for which 
$$\mu\sqsubset''\lambda\quad\Rightarrow \quad \ell(\mu)<\ell(\lambda).$$ 
Lemma~\ref{LemLength} therefore implies that $A(\lambda)\subseteq \nabla(\lambda)$, for all $\lambda\in\Lambda$. 
Also by Lemma~\ref{LemLength}, if this inclusion would not be an equality we need
$$\Ext^1(L(\nu),A(\lambda))\not=0,\quad\mbox{with $\ell(\nu)\le\ell(\lambda)$}.$$
However, \cite[Lemma~3.2(b)]{CPS} then implies that $\lambda \sqsubset''\nu$ which contradicts $\ell(\nu)\le\ell(\lambda)$.

In conclusion $A(\lambda)=\nabla(\lambda)$, and $\Ind\bC$ admits no highest weight structures which cannot be identified with its lfhw structure.

\section{Examples}\label{SecExam}

\subsection{Quasi-hereditary algebras}
The module categories of quasi-hereditary algebras are (lower) finite highest weight categories. 
Our uniqueness Theorem~\ref{ThmUniq} of highest weight structures thus applies to that setting. This gives an alternative to the proof for this fact from \cite{blocks}. The proof {\it loc. cit.} extends canonically to upper finite highest weight categories.

\subsection{Reductive groups}
Let $k$ be an algebraically closed field of positive characteristic $p$ and $G/k$ a reductive group. 

\begin{prop}\label{PropRepG}
The category $\Rep G$ admits precisely one highest weight structure, up to equivalence.
\end{prop}
\begin{proof}
As we will review below, it is well-known that every choice of maximal torus and positive root system yields a highest weight structure on $\Rep G$, see also \cite[Chapter II]{Jantzen}.
However, by \cite[Corollary~II.1.16 and \S II.2.12]{Jantzen}, every choice of maximal torus also yields an antiautomorphism $\tau:G\to G$ which equips $\Rep G$ with an involutive contravariant endofunctor which preserves simple representations. Uniqueness therefore follows from Theorem~\ref{ThmUniq}.
\end{proof}

\subsubsection{} Fix  a maximal torus $T$ of $G$, with the corresponding root system $R$ and character group $X(T)$. For a choice $R^+\subset R$ of positive system we have the dominant weights $\Lambda:=X(T)_+$, which label the simple modules in $\Rep G$ by \cite[\S II.2]{Jantzen}. For any $\lambda\in\Lambda$ we have the induced module $H^0(\lambda)={\mathrm{ind}}^G_B k_\lambda$ of \cite[\S II.2.1]{Jantzen} with simple socle $L(\lambda)$.  We denote by $\preceq$ the partial order on $\Lambda$ which is the restriction of the partial order $\uparrow$ on $X(T)$ of \cite[Ch II.6]{Jantzen}. It follows from \cite[Proposition II.6.20]{Jantzen} that $H^0(\lambda)$ is the injective hull of $L(\lambda)$ in $(\Rep G)_{\Omega}$ for any ideal $\Omega\subset\Lambda$ in which $\lambda$ is maximal. The Weyl module $V(\lambda):={}^\tau H^0(\lambda)$ (with $M\mapsto {}^\tau M$ the anti-equivalence of $\Rep G$ from \cite[\S II.2.12]{Jantzen}) is therefore the projective cover of $L(\lambda)$ in $(\Rep G)_{\preceq\lambda}$. That $(\Rep G,\preceq)$ is a lfhw category now follows from \cite[Proposition~II.4.13]{Jantzen} and the $\dagger$-version of Definition~\ref{DefHWC}.


\subsubsection{The essential order}

It is well-known that the essential order $\preceq^e$ on $X(T)_+$ of $\Rep G$ is different from (is coarser than) the partial order $\preceq$ described above, since the blocks of $\Rep G$ do not correspond to affine Weyl group orbits, see \cite[\S II.7.2]{Jantzen}. On the other hand, as we explain below, it is known that on the principal block $\preceq$ restricts to the essential order.

Following \cite{Jantzen} we denote the affine Weyl group by $W_p$, the finite Weyl group by $W<W_p$ and the $\rho$-shifted action of $w\in W_p$ on $\lambda\in X(T)$ by $w\bullet \lambda$. As the notation suggests, we view $W_p$ as a group of automorphisms of $X(T)$ such that the embedding $W_p<\mathrm{Aut}(X(T))$ depends explicitly on $p$. We also have a canonical realisation of $W_p$ as a Coxeter group with parabolic subgroup $W$.
We will always identify the quotient $W\backslash W_p$ with the corresponding subset of $W_p$ of shortest coset representatives Then we have a bijection
$$W\backslash W_p \,\to \, X(T)_+\cap W_p\bullet 0,  \quad w\mapsto w\bullet 0.$$
We denote the Bruhat order on $W\backslash W_p$, see \cite[\S 2.5]{BB}, by $\le_B$.
\begin{lemma}\label{LemJan}
For $x,y\in W\backslash W_p$ we have the following implications.
$$x\le_B y\quad\Rightarrow\quad x\bullet 0 \preceq^e y\bullet 0\quad\Rightarrow\quad x\bullet 0\uparrow y\bullet0.$$
\end{lemma}
\begin{proof}
The second implication follows by definition, so we prove the first one. By the chain property in \cite[Theorem~2.5.5]{BB}, it suffices to prove the special case where the difference in length $\ell(y)-\ell(x)$ is $1$.
Recall the function $d:W_p\to\mZ$ from \cite[\S 6.6]{Jantzen}. A reformulation of an observation made {\it loc. cit.} is that $\ell(x)=d(x)$ for $x\in W\backslash W_p$. It then follows from \cite[6.6(2)]{Jantzen} that $x\bullet 0\uparrow y\bullet 0$. Furthermore, \cite[Corollary~6.10]{Jantzen} implies that, for any $\nu\in X(T)$ with $x\bullet 0\uparrow \nu\uparrow y\bullet 0$, we have $\nu\in\{x\bullet 0,y\bullet 0\}$. It then follows from \cite[Corollary 6.24]{Jantzen} that $[\nabla(y\bullet 0):L(x\bullet0)]=1$, so in particular $x\bullet 0 \preceq^e y\bullet 0$.
\end{proof}

A priori, for $x,y \in W_p$, the relation $x\bullet 0\uparrow y\bullet 0$ depends on $p$. Assume that $p>h$, with $h$ the Coxeter number of \cite[6.2(9)]{Jantzen}.
Then it is known, see \cite[Lemma~10.1]{AR} and references therein, that the Bruhat order on $W\backslash W_p$ and $\uparrow$ coincide. Together with Lemma~\ref{LemJan} this allows us to conclude the folllowing.

\begin{corollary}
Assume that $p>h$. The essential order on $X(T)_+\cap W_p\bullet0$ of $(\Rep G)_0$ is given as follows.
For $x,y\in W\backslash W_p$ we have
$$x\le_B y\quad\Leftrightarrow\quad x\bullet 0 \preceq^e y\bullet 0\quad\Leftrightarrow\quad x\bullet 0\uparrow y\bullet0.$$
\end{corollary}
\subsection{Abelian Deligne category} Set $k=\mC$ and fix an arbitrary $t\in\mZ$. In \cite{EHS} an `abelian envelope' $\mathcal{V}_t$ of Deligne's universal symmetric monoidal category on a dualisable object $V_t$ of dimension $t$ was constructed. Using the universality in \cite[Theorem~2]{EHS}, it follows that we have a tensor functor $\cV_t\to \cV_t$ which sends $V_t$ to its (monoidal) dual $V_t^\ast$. The universality also shows that this functor is involutive. The composition of this functor with the monoidal duality $X\mapsto X^\ast$ is thus an anti-autoequivalence. It follows easily that this anti-autoequivalence preserves simple objects. Consequently, the lfhw structure on $\cV_t$ as defined in \cite[\S 8.5]{EHS} is the unique one which exists on $\cV_t$.

\subsection*{Acknowledgement}
The research was supported by ARC grant DE170100623. The author thanks Will Hardesty and Geordie Wlliamson for interesting discussions, Simon Riche for pointing out that Theorem~\ref{ThmInd}(1) is known and contained in \cite{KS} and Pramod Achar for posing the question addressed in Section~\ref{SecCPS}.

\end{document}